\documentclass{amsart}

\usepackage{amssymb, amsmath, amscd}

\usepackage{mathrsfs}

\usepackage[dvips]{color}

\usepackage{bbm}
\usepackage[active]{srcltx}
\usepackage{verbatim}

\usepackage{color}
\usepackage[colorlinks,hypertex]{hyperref}

\usepackage{newsymbol}
\let\emptyset \undefined
\let\ge       \undefined
\let\le       \undefined
\newsymbol\le          1336  
\newsymbol\ge          133E  
\newsymbol\emptyset    203F
\newsymbol\notle       230A
\newsymbol\notge       230B

\theoremstyle{plain}
\newtheorem{theorem}{Theorem}[section]
\theoremstyle{remark}
\newtheorem{remark}[theorem]{Remark}
\newtheorem{example}[theorem]{Example}
\theoremstyle{plain}

\newtheorem{lemma}[theorem]{Lemma}
\newtheorem{proposition}[theorem]{Proposition}

\newtheorem{assumption}[theorem]{Assumption}
\numberwithin{equation}{section}

\def\R{{\mathbb R}}
\def\C{{\mathbb C}}

\newcommand{\E}{{\mathbb E}}
\renewcommand{\P}{{\mathbb P}}
\newcommand{\F}{{\mathcal F}}
\newcommand{\calC}{{\mathscr C}}
\newcommand{\calA}{{\mathscr A}}
\newcommand{\calB}{{\mathscr B}}

\newcommand{\OO }{{\mathcal O}}

\renewcommand{\d}{\delta}
\newcommand{\e}{\varepsilon}
\newcommand{\la}{\lambda}
\renewcommand{\o}{\omega}

\renewcommand{\L}{L^2(E,\mui)}
\newcommand{\LH}{L^2(E,\mui;H)}
\newcommand{\LK}{L^2(\OO,\mui)}
\newcommand{\LKH}{L^2(\OO,\mui;H)}

\newcommand{\beq}{\begin{equation}}
\newcommand{\eeq}{\end{equation}}
\newcommand{\bal}{\begin{aligned}}
\newcommand{\eal}{\end{aligned}}
\newcommand{\ben}{\begin{enumerate}}
\newcommand{\een}{\end{enumerate}}
\newcommand{\bit}{\begin{itemize}}
\newcommand{\eit}{\end{itemize}}
\newcommand{\bea}{\begin{align}}
\newcommand{\eea}{\end{align}}
\newcommand{\bth}{\begin{theorem}}
\renewcommand{\eth}{\end{theorem}}
\newcommand{\bpr}{\begin{proposition}}
\newcommand{\epr}{\end{proposition}}
\newcommand{\ble}{\begin{lemma}}
\newcommand{\ele}{\end{lemma}}
\newcommand{\bpf}{\begin{proof}}
\newcommand{\epf}{\end{proof}}
\newcommand{\bex}{\begin{example}}
\newcommand{\eex}{\end{example}}
\newcommand{\bre}{\begin{example}}
\newcommand{\ere}{\end{example}}

\newcommand{\D}{{\mathsf D}}
\newcommand{\Ran}{\mathsf{R}}
\newcommand{\calL}{{\mathscr L}}
\newcommand{\n}{\Vert}
\newcommand{\one}{\mathbbm{1}}

\newcommand{\s}{^*}
\newcommand{\lb}{\langle}
\newcommand{\rb}{\rangle}

\newcommand{\limn}{\lim_{n\to\infty}}

\newcommand{\ra}{\rightarrow}
\newcommand{\ii}{\ensuremath{\infty}}
\newcommand{\wt}{\widetilde}

\newcommand{\wh}{\widehat}

\newcommand{\mui}{\mu_\infty}
\newcommand{\Qi}{Q_\infty}
\newcommand{\ovRD}{\overline{\Ran(D_\OO ^H)}}

\newcommand{\symm}{\hbox{\tiny\textcircled{s}}}

\begin{document}

\title[Non-symmetric Ornstein-Uhlenbeck semigroups on domains]{
$L^2$-Theory for non-symmetric Ornstein-Uhlenbeck semigroups on domains}

\author{Joyce Assaad}
\address{Notre Dame University - Louaize\\
Faculty of Natural and Applied Sciences
\\ Department of Mathematics and Statistics
\\ P.O. Box 72, Zouk Mikael
\\ Zouk Mosbeh
\\ Lebanon}
\email{joyce.assaad@ndu.edu.lb}

\author{Jan van Neerven}
\address{Delft Institute of Applied Mathematics
\\ Delft University of Technology 
\\ P.O. Box 5031
\\ 2600 GA Delft
\\ The Netherlands} 
\email{J.M.A.M.vanNeerven@TUDelft.nl}       

\begin{abstract}
We prove that the mild solution of the stochastic evolution equation
$dX(t) = AX(t)\,dt + dW(t)$
on a Banach space $E$ has a continuous modification if the associated 
Ornstein-Uhlenbeck semi\-group is analytic on $L^2$ with respect to the invariant measure.
This result is used to extend recent work of Da Prato and Lunardi for Ornstein-Uhlenbeck semigroups
on domains $\OO\subseteq E$ to the non-symmetric case. Denoting the generator of the 
Ornstein-Uhlenbeck semigroup by $L_\OO$, we obtain sufficient conditions
in order that the domain of $\sqrt{-L_\OO}$ be a first order Sobolev space.
\end{abstract}



\thanks{The second named author is supported by VICI subsidy 639.033.604
of the Netherlands Organisation for Scientific Research (NWO)}

\maketitle
 
\section{Introduction}
In this paper we present new results on analytic Ornstein-Uhlenbeck semigroups
associated with the linear Cauchy problem
$$ dX(t) = AX(t)\,dt + dW(t),$$
where $A$ is the generator of a $C_0$-semigroup on a Banach space $E$
and $W$ is a cylindrical Brownian motion,
and use them to extend recent work of Da Prato and Lunardi \cite{DL} for Ornstein-Uhlenbeck semigroups
on domains (see also \cite{DPGZ})
to the non-symmetric case. 
The approach in \cite{DL} is based on the Feynman-Kac formula
and uses the pathwise continuity of the Ornstein-Uhlenbeck process in a crucial way. Our first main
result (Theorem \ref{thm:cont-version}) asserts that in the non-symmetric case, pathwise
continuity still holds provided the Ornstein-Uhlenbeck semigroup is analytic on $L^2(E,\mu_\infty)$.
Here $\mu_\infty$ denotes an invariant measure whose existence we assume throughout.
Further new results concern the $\mu_\infty$-almost sure pointwise convergence of analytic
Ornstein-Uhlenbeck semigroups to the projection
onto the constant functions in $L^p(E,\mu_\infty)$ (Theorem \ref{thm:mu-as}) and a Poincar\'e
inequality for analytic Ornstein-Uhlenbeck semigroups (Theorem \ref{thm:PoincI}).

The construction and discussion of the main properties of the Ornstein-Uhlenbeck semigroup
$(P_\OO (t))_{t\ge 0}$ and its generator $L_\OO $
on an open domain $\OO$ in $E$, presented in Sections \ref{sec:FK} and \ref{sec:PK}, are
 extensions of their symmetric counterparts in \cite{DL}; see also \cite{DPGZ, Talar} for earlier work. 
In contrast,
the domain identification of $\D(\sqrt{-L_\OO })$ in $L^2(E,\mu_\infty)$ as a first order Gaussian Sobolev space is
essentially trivial in the symmetric case but requires substantial effort in the non-symmetric case.
In order to establish this identification, in Section \ref{sec:RT}
we adapt arguments from recent work by Maas and the second named
author \cite{MN2, MN-gradient}. 
As an application we prove a Poincar\'e inequality for the gradient on $\OO$ in the direction of $H$.

All spaces are real. The domain and range of a (possibly unbounded) linear operator $A$
are denoted by $\D(A)$ and $\Ran(A)$ respectively. Our terminology, in as far unexplained,
follows \cite{GN, MN, MN2, MN-gradient, NW}.

\section{Analytic Ornstein-Uhlenbeck semigroups}

Let $E$ be a real Banach space and $H$ a real Hilbert space, continuously embedded into $E$ by
means of a bounded injective
linear operator $i:H\to E$, and with inner product $[\cdot,\cdot]$. We fix a probability space $(\Omega,\P)$ and 
let $W^H$ an {\em $H$-cylindrical Brownian motion}, that is, a linear mapping $W^H: L^2(\R_+;H)\to L^2(\Omega,\P)$
satisfying
\begin{enumerate}
 \item for all $f\in L^2(\R_+;H)$ the random variable $W^H f$ is centred Gaussian distributed;  
 \item for all $f,g\in L^2(\R_+;H)$ we have 
$$ \E \lb W^H f \, W^H g) = \int_0^\infty [f(t),g(t)]\,dt.$$ 
\end{enumerate}
For $t\ge 0$ and $h\in H$ we put $$ W^H(t)h := W^H(\one_{[0,t]}\otimes h)$$
and note that $(W_H(t)h)_{t\ge 0}$ is a Brownian motion, which is standard if and only if $\n h\n = 1$.
Moreover, two such Brownian motions $(W_H(t)h)_{t\ge 0}$ and $(W_H(t)h')_{t\ge 0}$ are independent if and
only if $[h,h'] = 0$.   

Let $S = (S(t))_{t\ge 0}$ be a $C_0$-semigroup of bounded linear operators, with generator $A$, on $E$.
Throughout this paper we shall make the following standing assumption.

\begin{assumption}\label{ass:1}
The linear stochastic Cauchy problem
\begin{equation}\label{SPDE}
\begin{aligned}
dX(t) & = AX(t)\, dt + dW^H(t), \quad t\ge 0,
\end{aligned}
\end{equation}
admits an invariant measure.
\end{assumption}

Note that in \eqref{SPDE} we suppress the inclusion mapping $i$ and
identify $H$ with a linear subspace of $E$. Recall that a Radon measure $\mu$ on $E$ is said to be 
{\em invariant} for the problem \eqref{SPDE} if the following holds. Whenever $X_0$ is an $E$-valued
random variable, independent of $W^H$ and with distribution $\mu$, the initial value problem 
\begin{equation*}
\begin{aligned}
dX(t) & = AX(t)\, dt + dW^H(t), \quad t\ge 0, \\
 X(0) & = X_0,
\end{aligned}
\end{equation*}
is well-posed and its unique mild solution is stationary (with distribution $\mu$). 
Necessary and sufficient conditions for well-posedness can be found in \cite{NW}.

In order to arrive
at a useful equivalent formulation of Assumption \ref{ass:1} we need the following
terminology. The {\em reproducing kernel Hilbert space} $H_\mu$ associated with a centred Gaussian Radon
measure $\mu$ on $E$ is the closure in $L^2(E,\mu)$ of the dual space $E\s$ (identifying functionals $x\s\in E\s$ with the 
functions $x\mapsto \lb x,x\s\rb$ in $L^2(E,\mu)$). The mapping
$$i_\mu: H_\mu\to E, \qquad i_\mu x\s = \int_E \lb x,x\s\rb x\,d\mu(x)$$
is continuous and injective, and its adjoint is given by
$$i_\mu\s: E\s \to H_\mu, \qquad i_\mu\s x\s = \lb \cdot, x\s\rb.$$
Here and in what follows, we identify $H_\mu$ with its dual using the Riesz representation theorem.

Using this terminology, Assumption \ref{ass:1} is satisfied if and only if 
there exists a centred Gaussian Radon measure $\mui$ on $E$
whose reproducing kernel Hilbert space $i_\ii : H_\ii \to E$ satisfies
$$ \n i_\ii\s x\s\n_{H_\ii}^2 = \int_0^\infty \n i\s S\s(t)x\s\n_H^2\,dt, \quad x\s\in E\s.$$
(see  \cite{DZ, GN}).
The measure $\mu_\infty$ is then invariant. 

On the space $B_{\rm b}(E)$ of bounded Borel functions $f:E\ra\R$ we define the operators
$P(t)$, $t\ge 0$, by
\begin{equation}\label{eq:Pt}
 P(t)f(x):=\E f(X^x(t)), \quad t\ge0, \ \ x\in E,
\end{equation}
where
\begin{equation}\label{eq:mild} X^x(t) := S(t)x + \int_0^t S(t-s)\,dW^H(s)
\end{equation}

is the mild solution of the problem \eqref{SPDE}
with initial value $x$; the existence and uniqueness of this solution is implicit in the Assumption \ref{ass:1}.
These operators satisfy $P(0)=I$ and $P(t+s)=P(t)P(s)$ for all $t$ and $s\ge0$.
For all $1\le p<\infty$ the family $P = (P(t))_{t\ge 0}$ extends to a $C_0$-contraction
semigroup on $L^p(E,\mui)$ satisfying
\begin{align}\label{eq:invariance}
 \int_E P(t)f\,d\mui = \int_E f\,d\mui, \quad f\in L^p(E,\mui), \ t\ge 0.
\end{align}
Throughout this paper we make the following assumption.

\begin{assumption}\label{ass:2}
The semigroup $P$ is analytic on $L^2(E,\mui)$.
\end{assumption}

In statements like these, we always tacitly pass to the complexifications of the operators and
the spaces involved. Thus, what we are assuming is that $P_\C$ is analytic on $L^2(E,\mui;\C)$.
This assumption implies (see \cite{MN}) that $P_\C$ is in fact an analytic contraction semigroup on
$L^p(E,\mui;\C)$ for all $1<p<\infty$.

Necessary and sufficient conditions for this assumption to be satisfied are presented 
in \cite[Theorem 8.3]{GN}; this result 
extends previous results by Fuhrman \cite{Fu} and Go{\l}dys \cite{Go}.
As a corollary to this result (see \cite[Theorem 9.2]{GN}), 
Assumption \ref{ass:2} holds if the semigroup $S$ 
restricts to an analytic semigroup on $H$ which is {\em contractive with respect to some Hilbertian norm}.
This sufficient condition is close to being necessary: if Assumption \ref{ass:2} holds,
then $S$ restricts to a {\em bounded} analytic semigroup on $H$ \cite{MN-gradient}.

\begin{remark}
In applications to parabolic SPDEs the above sufficient condition is usually satisfied, the typical
situation being that $A$ is a second order elliptic operator on some domain $D\subseteq \R^d$ and $H = L^2(D)$.
\end{remark}

We proceed with a discussion of some consequences of Assumptions \ref{ass:1}
and \ref{ass:2} that will be needed later on.

Let $U: H_\ii\ra H$ be the linear operator with initial domain $i\s_\ii (E\s)$,
defined by $$U i\s_\ii x\s:=i\s x\s, \quad x\s\in E\s.$$
This operator is densely defined, and,
by \cite[Theorem 3.5]{GGN}, the analyticity of $P$ on $\L$ implies that $U$ is closable.
From now on we denote by $U$ its closure and by $\D(U)$ the domain of this closure.

Let $\phi:H_\ii\to \L$ be the isometric embedding given by
$$ (\phi (i_\ii\s x\s)) (\cdot):= \lb \,\cdot\,,x\s\rb, \quad x\s\in E\s.$$
In order to simplify notations a bit, we shall write $$\phi_h(x) := (\phi (h))(x).$$
When $H_0$ is a linear subspace of $H_\ii$ and $k\ge0$ is an integer, we denote by $\F
C_{\rm b}^k(E,H_0)$ the vector space of all $\mui$-almost everywhere defined
functions $f:E\ra\R$ of the form
\[f:=\varphi(\phi_{h_1},\cdots,\phi_{h_n}), \quad x\in E,\]
with $n\ge1, \varphi\in C_{\rm b}^k(\R^n)$, and $h_1,\cdots, h_n\in H_0$. Here
$C_{\rm b}^k(\R^n)$ is the space of all bounded continuous functions with bounded
continuous derivatives up to order $k$.\\
For $f\in \F C_{\rm b}^1(E, \D(U))$ we define the Fr\'echet derivative $D^H f:E\to H$ of $f$ in the direction of
$H$ by
\[D^Hf:=\sum_{j=1}^{n} \partial_j\varphi(\phi_{h_1},\cdots,
\phi_{h_n})\otimes Uh_j.\]
The closability of $U$ implies that $D^H$ is closable from $L^p(E,\mui)$ to $L^p(E,\mui;H)$ for all
$p\in [1,\ii)$ (see \cite[Theorem 3.5]{GGN} and \cite[Proposition 8.7]{GN}). Henceforth, by slight abuse of notation
 we denote by $D^H$ its closure in $L^p(E,\mui)$ and write
$$W_H^{1,p}(E,\mui):= \D_p(D^H)$$
for the domain of this closure in $L^p(E,\mui)$. Furthermore we write $$\D(D^H) := \D_2(D^H).$$

Under Assumption \ref{ass:1}, $S$ maps $H_\ii$ into itself, and the restriction $S_\infty = S|_{H_\ii}$
is a $C_0$-contraction semigroup on $H_\ii$. We shall denote its generator by $A_\infty$.
The next result is taken from \cite{GN} and \cite{MN}.

\begin{proposition}\label{prop:MN} Let Assumptions \ref{ass:1} and \ref{ass:2} be satisfied.
There exists a unique bounded operator $B\in \calL(H)$ such
that $$i B i\s x\s= - i_\ii A_\infty\s i_\ii\s x\s, \quad x\s \in \D(A\s).$$
This operator satisfies  $B+B\s=I$ and
 $[Bh,h] =\frac1 2\n h\n_H^2$ for all $h\in H$.
\end{proposition}

Let $l$ be the sesquilinear form defined by
 \begin{equation}\label{eq:l}
l(f,g):=[BD^Hf,D^Hg]\quad f,g\in \D(l):=\D(D^H),
 \end{equation}
where $B\in \calL(H)$ is the bounded operator described in Proposition \ref{prop:MN}.
 This form is closed, densely defined, accretive and sectorial on
 $L^2(E,\mui)$. Let us denote by $D^{H*} BD^H$ the associated sectorial operator with domain consisting
 of all functions $f\in \D(D^H)$ such that $BD^Hf \in \D(D^{H*})$.  This domain is a core
 for $\D(D^H)$ (see \cite[Lemma 1.25]{O}) and we have (see \cite{MN}):

\begin{proposition}\label{prop:div}
 Let Assumptions \ref{ass:1} and \ref{ass:2} be satisfied.
The generator $L$ of the semigroup $P$ on $\L$ equals $$L = - D^{H*} BD^H.$$
\end{proposition}

Below we shall also need the following result.

\begin{theorem}\label{thm:cont-version}
 Let Assumptions \ref{ass:1} and \ref{ass:2} be satisfied.
Then, for any initial condition $x_0\in E$, the mild solution
$X^{x_0}$ of the problem \eqref{SPDE} admits a continuous modification.
\end{theorem}
\begin{proof}
Without loss of generality we may assume that $x_0=0$.

On a possibly larger probability space,
let $Y_\infty$ be a centred $E$-valued Gaussian random variable,
independent of $W^H$, with distribution
$\mui$. Then the process $Z=(Z(t))_{t\ge 0}$ defined by
$ Z(t) = S(t)Y_\infty + X^0(t)$
is stationary; as before $X^0$ denotes the mild solution with initial value $0$.
By the strong continuity of the semigroup $S$, the process $X^0$ has a continuous
modification if and only if this is true for $Z$.

For $t \ge s\ge 0$ we have
$$
\begin{aligned}
\E \lb Z(t), x\s\rb\lb Z(s), x\s\rb & =
\E \lb Z(0), x\s\rb\lb Z(t-s), x\s\rb \\ & =
\E \lb Y_\infty, x\s\rb\lb S(t-s)Y_\infty + X^0(t-s),x\s\rb \\ &  =
\E \lb Y_\infty, x\s\rb\lb S(t-s)Y_\infty, x\s\rb \\ & =
\lb \Qi S\s(t-s)x\s, x\s\rb,
\end{aligned}
$$
where we used that $Y_\infty$ and $X^0(t)$ are independent for every $t\ge 0$.

For $t\ge s\ge 0 $, by the above we have
$$
\begin{aligned}
0\le  \E \lb Z(t) - Z(s),x\s\rb^2
& = \E \lb Z(t),x\s\rb^2 + \E \lb Z(s),x\s\rb^2 -
 2 \E \lb Z(t), x\s\rb\lb Z(s),x\s\rb
\\ &  = 2\lb \Qi x\s, x\s\rb - 2 \lb \Qi x\s, S\s(t-s)x\s\rb
\\ &  = 2\lb \Qi (I - S\s(t-s))x\s,x\s\rb.
\end{aligned}
$$
The authors would like to thank Ben Go{\l}dys for showing this argument.

Since $\mui$ is a Radon measure, the closure $E_0$ of its reproducing kernel Hilbert space
$H_\infty$ in $E$ is
separable,
and we have $\mui(E_0)=1$. The invariance of $H_\infty$ under $S$ (see \cite{CG, Nee}) implies
that
also $E_0$ is invariant under $S$.
The covariance operator $Q_t$ of $\mu_t$ satisfies
$$\lb Q_t x\s,x\s\rb  = \int_0^t \n i\s S\s(s)x\s\n_H^2\,ds \le \lb \Qi x\s,x\s\rb, \quad x\s\in E\s,$$
and therefore Anderson's inequality implies
that $\mu_t(E_0) =1$ for all $t\ge 0$.
It follows that for all $t\ge 0$, $X_t\in E_0$ almost surely.
Since $H$ is contained in $E_0$ (by \cite[Proposition 2.6]{GN}),
this argument shows that without loss of generality we may assume that $E$ is
separable.

The analyticity of $P$
on $L^2(E,\mui)$ implies that the operator
$\Qi A\s x\s,$ which is
well defined on the domain $\D(A\s)$,
extends to a bounded operator from $E\s$ to $E$. In fact, we have
$$ \n \Qi A\s x\s\n \le \n i\n \n \Qi A\s x\s\n_H \le \n i\n \n B\n \n i\s x\s\n_H \le \n i\n^2 \n B\n \n x\s\n.$$
By a standard argument, this implies that for all $x\s\in E\s$,
\begin{equation}\label{eq:1}
 \E \lb Z(t) - Z(s),x\s\rb^2  =  2\lb \Qi(I-S\s(t-s))x\s, x\s\rb
\le M_T|t-s| \n i\n^2 \n B\n \n x\s\n^2,
\end{equation}
where $M_T = \sup_{0\le t\le T} \n S(t)\n$.
The process $\lb Z,x\s\rb$ being Gaussian, the Kolmogorov continuity criterion then
implies that
the process $\lb Z, x\s\rb$ has a continuous
modification.
By \eqref{eq:1} and the stationarity of $Z$, the conditions of \cite[Proposition 1]{Ca} are satisfied
and we conclude that the Gaussian process $(\lb Z(t),x\s\rb)_{(t,x\s)\in
[0,T]\times B_{E\s}}$ has a continuous modification.
The existence of a continuous modification of $(Z(t))_{t\in [0,T]}$ then follows from
\cite[Theorem 1.2]{Fer}.
\end{proof}

\begin{remark}
The problem of existence of a continuous version for the mild solution of \eqref{SPDE}
has been discussed by many authors. If the inclusion mapping $i:H\to E$ is $\gamma$-radonifying
(if $E$ is a Hilbert space, this is equivalent to $I$ being Hilbert-Schmidt),
a continuous version exists if $E$ has type $2$; this follows by the factorization method of
Da Prato, Kwapie\'n, and Zabczyk \cite{DKZ}. For Hilbert spaces $E$, the result is due to Smole\'nski \cite{Sm};
the type $2$ case follows from Millet and Smole\'nski \cite{MS} combined with a result of Rosi\'nski and Suchanecki
\cite{RS} (see also \cite{NW}).
The special case with $E$ a Hilbert space had been treated before by \cite{Sm}.   
In the general case considered here ($E$ an 
arbitrary Banach space, $i:H\to E$ bounded and injective) a continuous version exists
if $S$ is analytic on $E$ \cite{BN}. Analyticity of $P$ does not in general imply
analyticity of $S$ (a counterexample can be found in  \cite{MN-gradient}), so our Theorem
\ref{thm:cont-version}
is not contained as a special case in the result in \cite{BN}. In the converse direction we mention that 
neither does the analyticity of $S$ imply that of $P$; a counterexample is due to Fuhrman \cite{Fu}.
\end{remark}

\begin{remark}
Examples of `Ornstein-Uhlenbeck like' processes without continuous version are presented in 
\cite{DKZ, Iscoe}; in both references, these processes arise as mild solutions of an equation
of the form $$dU = AU(t)\,dt + B\,dW^H$$
with an {\em unbounded} densely defined closed linear operator $B: \D(B)\subseteq H\to E$.
\end{remark}

In the rest of this paper, we will always work with a continuous version of $X$
whose existence is guaranteed by Theorem \ref{thm:cont-version}.

We continue with two almost everywhere convergence results.
The first concerns the behaviour of $P(t)$ as $t\downarrow 0$. It follows from
the $L^p$-bound\-ed\-ness of the maximal function
$$ Mf(x):= \sup_{t>0} |P(t)f(x)|;$$
see \cite{Cow83}
 and \cite[Proposition 8.5]{MN2} (where the present setting
is considered).

\begin{theorem} Let Assumptions \ref{ass:1} and \ref{ass:2} be satisfied and let $1<p<\infty$.
For all $f\in
L^p(E,\mu_\infty)$ we have
$$ \lim_{t\downarrow 0} P(t)f(x) = f(x) \ \ \hbox{ for $\mu_\infty$-almost all $x\in E$.}
$$
\end{theorem}

The second result concerns the behaviour of $P(t)$ as $t\to\infty$.
Below we shall only need the part (1) (with $p=2$) (see also
\cite[Proposition 10.1.1]{DZ1} for a partial result in this direction).

\begin{theorem}\label{thm:mu-as}
Fix $1\le p<\infty$.
\ben
\item[\rm(1)] If Assumption  \ref{ass:1} is satisfied, then
for all $f\in L^p(E,\mu_\infty)$ we have
$$ \lim_{t\to\infty} P(t)f = \int_E f\,d\mu_\infty \ \ \hbox{in $L^p(E,\mu_\infty)$.}$$
\item[\rm(2)] If Assumptions \ref{ass:1} and \ref{ass:2} are satisfied,
then for all $f\in L^p(E,\mu_\infty)$ we have
$$ \lim_{t\to\infty} P(t)f(x) = \int_E f\,d\mu_\infty \ \ \hbox{
for $\mu_\infty$-almost all $x\in E$.}$$
\een
\end{theorem}

\begin{proof}
The proof of the first statement follows by second quantisation
and using the fact
\cite[Proposition 2.4]{GN} that
$S_\infty\s$ is strongly stable. The details are as follows.

First we consider the case $p=2$. For all $h_1,\dots,h_n \in H_\infty$
we have
$$\lim_{t\to\infty} (S_\infty^*(t))^{\otimes n} (h_1\otimes\dots\otimes h_n)
= \lim_{t\to\infty} S_\infty^*(t)h_1\otimes\dots\otimes S_\infty^*(t) h_n = 0,
$$
from which it follows that $(S_\infty^*)^{\otimes n} := S_\infty^*\otimes\dots\otimes
S_\infty^*$ ($n$ times)
is strongly stable on
$H_\infty^{\otimes n}:= H_\infty\otimes\dots\otimes H_\infty$. By restricting to the symmetric tensor products
$H_\infty^{\symm n}$  and taking direct sums, it follows that
the second quantised semigroup
$$\Gamma(S_\infty\s) := \bigoplus_{n=0}^\infty (S_\infty^{\symm n})^*$$
is strongly stable on the closed subspace $ \bigoplus_{n=1}^\infty
H_\infty^{\symm n}$ of  $\bigoplus_{n=0}^\infty H_\infty^{\symm n}$. 
Under the Wiener-It\^o isometry, the latter space is mapped isometrically onto
$L^2(E,\mu_\infty)$, and the first summand $H_\infty^{\symm 0}$ is mapped
onto the one-dimensional subspace spanned by the constant function
$\bf{1}$. Moreover, under this isometry the semigroup
$\bigoplus_{n=0}^\infty (S_\infty^{\symm n})^*$ corresponds to $P$
 in the sense that the following diagram commutes:
\begin{equation*}
 \begin{CD}
    \bigoplus_{n=0}^\infty
H_\infty^{\symm n}     @> \bigoplus_{n=0}^\infty (S_\infty^{\symm n})^* >>   \bigoplus_{n=0}^\infty
H_\infty^{\symm n} \\
     @V  \simeq  VV              @VV \simeq V \\
     L^2(E,\mui)     @> P(t) >> L^2(E,\mui) 
     \end{CD}
\end{equation*}
the isomorphism on the vertical arrows being the Wiener-It\^o isomorphism
(see \cite{CG, Nee}). As a result, we find that the semigroup $P$ is
strongly stable on $L^2(E,\mu_\infty)\ominus \R \bf{1}$. Since
$P(t)\bf{1} = \bf{1}$ and $(\int_E f\,d\mu_\infty) \bf{1}$ equals the
orthogonal projection of $f$ onto $\R\bf{1}$, this gives the first
assertion for $p=2$.

Next let $2<p<\infty$ be arbitrary, and choose $p<q<\infty$ arbitrarily.
Since $P$ is contractive on $L^q(E,\mu_\infty)$,
for all $f\in L^q(E,\mu_\infty)$ we have, by convexity,
$$\bal
\Big\n P(t)f - \int_E f\,d\mu_\infty\Big\n_p
& \le \Big\n P(t)f - \int_E f\,d\mu_\infty\Big\n_2^{1-\theta}
\Big\n P(t)f - \int_E f\,d\mu_\infty\Big\n_q^\theta
\\ & \le \Big\n P(t)f - \int_E f\,d\mu_\infty\Big\n_2^{1-\theta} (2\n f\n_q)^\theta,
\eal$$
where $0<\theta<1$ satisfies $\frac{1-\theta}{2} + \frac{\theta}{q} = \frac1p.$
The right-hand side tends to $0$ as $t\to\infty$.
Since  $L^q(E,\mu_\infty)$ is dense in
$L^p(E,\mu_\infty)$ and $P$ is contractive on $L^p(E,\mu_\infty)$,
this implies the first assertion for $2<p<\infty$.

Next let $1\le p<2$. For $f\in L^2(E,\mu_\infty)$, the
$L^2$-convergence implies the $L^p$-convergence by H\"older's inequality.
Since $L^2(E,\mu_\infty)$ is dense in $L^p(E,\mu_\infty)$ and $P$
is contractive  on $L^p(E,\mu_\infty)$, this gives the first
assertion for $1\le p<2$.

For the proof of (2) we fix $1<p<\infty$. We shall identify a dense
subspace of functions for which the asserted $\mu_\infty$-almost everywhere
convergence does hold. By the $L^p$-boundedness of
the maximal function, which follows from the analyticity
of $P$ by \cite[Proposition 8.5]{MN2},
the set of all functions for which we have
$\mu_\infty$-almost everywhere
convergence is norm-closed in $L^p(E,\mu_\infty)$ and the proof is complete.

For $h\in H_\infty$ define $$K_h :=
\exp\Big(\phi_h - \frac12\n h\n_{H_\infty}^2\Big).$$
As is well-known, these functions belong to $L^p(E,\mu_\infty)$ and their
linear span is dense in $L^p(E,\mu_\infty)$.
Moreover, from the identity
$$K_h = \sum_{n=0}^\infty \frac1{n!} I_n (\phi_h^n),$$
with $I_n$ the orthogonal projection in $L^2(E,\mu_\infty)$ onto the
$n$-th Wiener-It\^o chaos, it follows that
$$\int_E {K_h}\,d\mu_\infty = 1.$$
By second quantisation,
$$ P(t) K_h = K_{S_\infty\s (t)h}.$$
The proof will be finished by observing that for $\mu_\infty$-almost
all $x\in E$ we have
$$ \lim_{t\to\infty} P(t) K_h (x) = \lim_{t\to\infty}
\exp\Big(\phi_{S_\infty\s (t)h}(x) - \tfrac12\n S_\infty\s(t)h\n_{H_\infty}^2\Big)
= 1 = \int_E {K_h}\,d\mu_\infty.
$$
In this computation we used that
$\lim_{t\to\infty} S_\infty\s(t)h = 0$ in $H_\infty$,
from which we shall deduce next that $\lim_{t\to\infty}\phi_{S_\infty\s(t)h} = 0$
$\mu_\infty$-almost surely. Once this has been shown the proof is complete.

We start by noting that $$P(t)\phi_h = \phi_{S_\infty\s(t)h}.$$
Hence by the $L^2$-boundedness of the maximal function,
$$ \big\n \sup_{t>0} |\phi_{S_\infty\s(t)h}|\big\n_{L^2(E,\mu_\infty)} \lesssim
\n h\n_{H_\infty}.$$
By the semigroup property, this implies that
$$ \big\n \sup_{t>T} |\phi_{S_\infty\s(t)h}|\big\n_{L^2(E,\mu_\infty)}
 =  \big\n \sup_{t>0} |\phi_{S_\infty\s(t+T)h}|\big\n_{L^2(E,\mu_\infty)}
 \lesssim \n S_\infty\s(T) h\n_{H_\infty}. $$
The right hand side of this expression tends to $0$ as $T\to\infty$.
Having observed this, the proof can be finished with a standard Borel-Cantelli
argument.
With Chebyshev's inequality we find times $T_n\to \infty$ such that
$$ \mu_\infty \Big(x\in E: \ \sup_{t>T_n} |\phi_{S_\infty\s(t)h}(x)| > \frac1{2^n}\Big) < \frac1{2^n}.$$
By the Borel-Cantelli lemma it follows that
$$\mu_\infty \Big(x\in E: \ \sup_{t>T_n} |\phi_{S_\infty\s(t)h}(x)| > \frac1{2^n} \ \hbox{for
infinitely many $n$}\Big) = 0.$$
Hence for $\mu_\infty$-almost all $x\in E$ we can find $n_0$ (depending on $x$) such that
$$ \sup_{t>T_n} |\phi_{S_\infty\s(t)h}(x)| \le  \frac1{2^n} \ \ \hbox{for all
$n\ge n_0$}.$$
Clearly that implies that $\lim_{t\to\infty} \phi_{S_\infty\s(t)h}(x) = 0$
for $\mu_\infty$-almost all $x\in E$.
\end{proof}

The next result is an extension of \cite[Theorem 3.3, Corollary 3.4]{CG2000},
where the stronger assumption was made
that $\n S_\infty(t)\n \le e^{-w t}$ for some $w>0$ and all $t\ge 0$.

\bth[Poincar\'e inequality] \label{thm:PoincI}
Let Assumptions \ref{ass:1} and \ref{ass:2} be satisfied.
If the semigroup $S_\infty$ is uniformly exponentially stable, then there is a constant $C$ such that
for all $\phi\in W_H^{1,2}(E,\mu_\infty)$ we have
$$ \int_E (\phi - \overline{\phi})^2\,d\mu_\infty \le C \int_E \n D^H\phi\n^2\,d\mu_\infty.$$
\eth

\begin{proof}
By \cite{MN-gradient}, $S$ restricts to a bounded analytic $C_0$-semigroup $S_H:= S|_H$
on $H$, and by \cite[Theorem 5.4]{GN} this semigroup is uniformly exponentially stable,
say $\n S_H(t)\n \le M e^{-w t}$ with $M\ge 1$ and $w>0$.

Next we note (see \cite[Theorem 5.6]{MN2}) that $P(t)f\in W_H^{1,2}(E,\mu_\infty)$ and
$ D^H P(t) f =  (P(t) \otimes S_H(t))D^H f.$ Hence, $\mu_\infty$-almost everywhere we have
\begin{equation}
 \label{eq:square}
 \n D^H P(t) f\n_H ^2 = \n P(t)\otimes S_H(t) D^H f\n_H^2 \le M^2e^{-2wt}  P(t) (\n D^H f\n_H^2).
\end{equation}
Combining \eqref{eq:square} with Proposition \ref{prop:div} and Theorem \ref{thm:mu-as},
as in \cite[Proposition 2.2(a)] {DL}
the desired result follows a method of Deuschel-Stroock \cite{DS} (following the lines of the proof of
\cite[Proposition 10.5.2]{DZ1}, using the expression for $L$ as given in \cite{MN}; this produces
the constant $M^2/2w$).
\end{proof}

\section{The Feynman-Kac semigroup on $\L$}\label{sec:FK}

In this section and the next, we extend the results of \cite[Section 3]{DL}
to the non-symmetric setting. Our proofs follows those of \cite{DL} closely,
with some modifications necessitated by the non-selfadjointness of $L$.
Another subtle difference concerns the assumptions on the domain $\OO$, which we
take to be open as in \cite{DPGZ, Talar}; 
in \cite{DL} closed domains are considered (in this connection see also Remark \ref{rem:comparison}).  
For the convenience of the reader (and for the sake of mathematical rigour) 
we have therefore decided to write out all proofs in detail.

We shall always assume that Assumptions \ref{ass:1} and \ref{ass:2}
are satisfied without repeating this at every instance.
We fix an nonempty open subset $\OO  $ in $E$ satisfying
 \[\mui(\OO)>0 \] and
a bounded continuous function  $V:E\to [0,1]$
which satisfies
\begin{equation}\label{V}
\begin{cases}
V(x)=0, & \text{$x \in \overline{\OO}$,}\\
V(x)>0, & \text{$x \in \complement \overline{\OO}$.}
\end{cases}
\end{equation}

For $f\in B_{\rm b}(E)$,  $x\in E$, and $\e>0$ set
\begin{equation}\label{eq:def-P_eps}
P_\e(t)f(x):=\E\big[f(X^x(t))e^{-\frac1 \e\int_0^tV(X^x(r))dr}\big].
\end{equation}
By standard arguments,
$P_\e = (P_\e(t))_{t\ge 0}$ is a semigroup
of linear contractions on $B_{\rm b}(E)$, the so-called {\em Feynman-Kac semigroup} associated
with $-L+\tfrac1{\e}V$.

\bpr[cf. \hbox{\cite[Proposition 3.1]{DL}}]
For all $f\in B_{\rm b}(E)$ and $\e>0$,
 \[\int_E(P_\e(t)f)^2\,d\mui\le\int_Ef^2\,d\mui.\]
As a consequence, $P_\e$ is uniquely extendable to a $C_0$-semigroup of contractions on $\L$.
\epr

\bpf
Using the Cauchy-Schwarz inequality, for $f\in B_{\rm b}(E)$ we have
\begin{align*}
(P_\e(t)f(x))^2&=(\E[f(X^x(t))e^{-\frac1 \e\int_0^tV(X^x(r))dr}])^2\\
&\le\E[f^2(X^x(t))e^{-\frac2 \e\int_0^tV(X^x(r))dr}]
\\ & \le\E[f^2(X^x(t))] = P(t)f^2(x).
\end{align*}
Integrating with respect to $\mui$ and using
\eqref{eq:invariance}, we obtain
\[
\int_E(P_\e(t)f)^2\,d\mui\le\int_EP(t)f^2\,d\mui=\int_Ef^2
\,d\mui.\]
This shows that the operators $P_\e(t)$ are contractive on $\L$. To see that
the resulting semigroup $P_\e$ is strongly continuous, note that
for all $f\in C_{\rm b}(E)$ the mapping $t\mapsto P_\e(t)f(x)$
is continuous for each $x\in E$ by the path continuity of $t\mapsto X^x(t)$.
Hence, by dominated convergence, $\lim_{t\downarrow 0} P_\e(t)f = f$ in $\L$
for all $f\in C_{\rm b}(E)$. By density and uniform boundedness, the
strong continuity of $P_\e$ follows from this.
\epf

From now on, unless stated otherwise, we shall denote by $P_\e$ the $C_0$-semigroup of
contractions on $L^2(E,\mui)$ whose existence is assured by the above proposition.
Our next aim is to identify $L-\frac 1 \e V$ as its generator.

For fixed $\la>0$ and $f\in \L$, let us consider the resolvent equation
\beq\label{eq}
\la\phi_\e - L\phi_\e+\frac1 \e V\phi_\e=f.
\eeq

\bpr[cf. \hbox{\cite[Proposition 3.2]{DL}}]
The equation (\ref{eq}) has a unique solution $\phi_\e\in \D(L)$ and the
following estimates hold:
\begin{align}\label{R}
\int_E\phi_\e^2\,d\mui & \le \frac1 {\la^2}\int_Ef^2\,d\mui,
\\
\label{DR}
\int_E\n D^H\phi_\e\n_H^2\,d\mui & \le\frac 2 \la \int_Ef^2\,d\mui,
\\
\label{VR}
\int_{E }\phi_\e^2 V\,d\mui & \le\frac \e \la \int_Ef^2\,d\mui,
\\
\label{DV}
\int_{E }\n D^H\phi_\e\n_H^2V \,d\mui & \le\frac {\e^{1/2}} {\la^{1/2}} \int_Ef^2\,d\mui.
\end{align}
\epr
\bpf
We know that the form $l$ defined in \eqref{eq:l} is closed, densely defined, sectorial
and accretive.  Since
\[|[\frac1 \e Vf,f]_\L|\le
\frac 1 \e \n V \n_\ii\n f\n^2_{\L} ,\]
the KLMN theorem (see \cite[Theorem VI.1.33]{K}) shows that the form associated
to $-L+\frac1 \e V$ is closed, densely defined, and sectorial. It is also
accretive since $-L+\frac 1 \e V \ge -L\ge 0$. Therefore, $-L+\frac 1 \e V$ is
maximal accretive,
and (\ref{eq}) has a unique solution $\phi_\e\in \D(L)$.
Thus
\[\int_E\phi_\e^2\,d\mui=\big\n (\la-L+\frac1\e V)^{-1}f\big\n^2_{\L}\le
\frac1{\la^2}\n f\n_{L^2(E,\mui)}^2 =
 \frac1
{\la^2}\int_Ef^2\,d\mui.\]
Let us now multiply  both sides of (\ref{eq}) by $\phi_\e$ and
integrate over $E$:
\beq\label{eq2}
\la\int_E\phi_\e^2\,d\mui-\int_E L\phi_\e\cdot \phi_\e \,d\mui+\frac1 \e
\int_EV\phi_\e^2\,d\mui=\int_Ef\phi_\e \,d\mui.
\eeq
Since $[Bu,u]=\frac1 2\n u \n_H^2,$
 \[-\int_E L\phi_\e\cdot \phi_\e\,d\mui=\int_E[BD^H\phi_\e,D^H\phi_\e]
\,d\mui=\frac12\int_E\n D^H\phi_\e\n_H^2\,d\mui.\] Substituting this identity in
(\ref{eq2}) yields
\begin{align*}
\frac 1 2 \int_E\n D^H\phi_\e\n_H^2\,d\mui&\le \int_Ef\phi_\e \,d\mui\\
&\le\Big(\int_Ef^2\,d\mui\Big)^{1/2}\Big(\int_E\phi_\e^2\,d\mui\Big)^{1/2}
\le\frac 1 \la \int_Ef^2\,d\mui,
\end{align*}
where we used the Cauchy-Schwarz inequality and (\ref{R}). 
 
We also notice from
(\ref{eq2}) that
\[\frac 1 \e \int_EV\phi_\e^2\,d\mui\le\int_E f\phi_\e \,d\mui\le\frac 1 \la
\int_Ef^2\,d\mui.\]

Next, multiplying  both sides of (\ref{eq}) by $\phi_\e V$ and integrating gives
\beq\label{eq2a}
\la\int_E\phi_\e^2 V\,d\mui-\int_E L\phi_\e\cdot \phi_\e V\,d\mui+\frac1 \e
\int_EV\phi_\e^2V\,d\mui=\int_Ef\phi_\e V\,d\mui.
\eeq
Repeating the reasoning following \eqref{eq2}, from \eqref{eq2a} we infer
\begin{align*}
\frac 1 2 \int_E\n D^H\phi_\e\n_H^2 V\,d\mui&\le \int_Ef\phi_\e V\,d\mui\\
&\le\Big(\int_Ef^2V\,d\mui\Big)^{1/2}\Big(\int_E\phi_\e^2V\,d\mui\Big)^{1/2} \\
&\le\Big(\int_Ef^2\,d\mui\Big)^{1/2}\Big(\int_E\phi_\e^2V\,d\mui\Big)^{1/2}
\le\frac {\e^{1/2}} {{\la}^{1/2}} \int_Ef^2\,d\mui.
\end{align*}
\epf

To prove that $L-\frac 1\e V $ is the generator of the semigroup
$(P_\e(t))_{t\ge0}$ we need the following result. Let
$$ \calC:={\rm span}\{P(t)f: \ t>0, \  f\in
C_{\rm b}(E)\}.$$

\ble\label{calC}
$\calC$ is a core for $\D(L)$ and we have $\calC \subseteq \D(L)\cap C_{\rm b}(E)$.
\ele
\bpf
Since $C_{\rm b}(E)$ is dense in $\L$ and contained in $\overline{\calC}$,
 $\calC$ is dense in $\L$. Since $P$ is analytic on $\L$, $\calC$ is
contained in $\D(L)$. Moreover $\calC$ is
$P$-invariant, and therefore $\calC$ is a core for $\D(L)$.
Finally, it is immediate from \eqref{eq:Pt} that $P(t)f\in C_{\rm b}(E)$
for all $t>0$ and $f\in C_{\rm b}(E)$, so $\calC\subseteq \D(L)\cap C_{\rm b}(E)$.
\epf

Let $M_\e$ be the infinitesimal generator of $P_\e$ on $L^2(E,\mui)$.

\bpr[cf. \hbox{\cite[Proposition 3.3]{DL}}] We have $\D(M_\e)=\D(L)$
and
\beq\label{Ge}
M_\e =L-\frac1\e V.
\eeq
\epr

\bpf
Let us show that $\D(L)\subseteq \D(M_\e)$
and that the identity (\ref{Ge}) holds on $\D(L)$. Then, since
both $M_\e$ and $L-\frac1\e V$ are semigroup generators, the identity $\D(M_\e)=\D(L)$ follows.

Fix $f\in
\D(L)\cap C_{\rm b}(E)$ and $x\in E$. For all $t>0$,
\begin{equation*}
\begin{aligned}
P_\e(t)f(x)-f(x)&= \E[f(X^x(t))e^{-\frac1\e\int_0^tV(X^x(s))ds}]-f(x)\\
&= \E[f(X^x(t))]-f(x)+\E[f(X^x(t))(e^{-\frac1
\e\int_0^tV(X^x(s))ds}-1)].
\end{aligned}
\end{equation*}
Dividing
both sides by $t$ and letting $t\downarrow 0$, by pathwise continuity and
dominated convergence we obtain
$$\frac1t(P_\e(t)f-f)\ra  Lf-\frac1 \e Vf$$
in $\L$. It follows that
$f\in \D(M_\e)$ and $M_\e f=Lf-\frac 1 \e V f$.

Let now $f\in \D(L)$ be arbitrary. Let $f_n\ra f$ in $\D(L)$ with $f_n\in \calC$, where $\calC$
is as in Lemma \ref{calC}.
Then  $f_n\ra f$ in $\L$, $ Lf_n\ra Lf$ in $\L$,
and $\frac 1 \e Vf_n\ra \frac 1
\e Vf$ in $\L$. Therefore $M_\e f_n=Lf_n-\frac1 \e V
f_n\ra Lf-\frac1 \e V f$ in $\L$. Since $M_\e$ is closed, this implies that $f\in \D(M_\e)$ and
$M_\e f=Lf-\frac 1 \e V f$.
\epf

\section{The semigroup $P_\OO (t)$}\label{sec:PK}

On $B_{\rm b}(\OO)$, following \cite{DPGZ, Talar} we define the
operators $P_\OO (t)$ for $t\ge0$ by 
\[P_\OO (t)f(x):=\E[f(X^x(t))\one_{\{\tau^x_{{\OO}} > t\}}], \quad x\in\OO.\]
Here, \begin{align*}\tau^x_\OO :=\inf\{t>0:\  X^x(t)\in \complement 
\OO\}
\end{align*} is the entrance time of $\complement \OO$
corresponding to the initial value $x$.
As $\tau_\OO^x>0$ for all $x\in\OO$ it is clear that $P_\OO(0)f=f$, and an easy calculation 
based on \eqref{eq:mild} 
shows that $P_\OO(t)P_\OO(s)f = P_\OO(t+s)f$ for all $t,s\ge 0.$

For $\e>0$ let $$\OO_\e := \{x\in \OO: \ d(x,\complement \OO) > \e\}.$$
Let $V_\e: \E \to [0,1]$ be the potential defined by
$$ V_\e (x) = \tfrac1\e d(x,\OO_\e) \wedge 1.$$
Note that $V_\e \equiv 0$ on $\overline{\OO}_\e$ and $V_\e \equiv 1$ on $\complement \OO$. 
In the results below, we denote by $P_\e$ the strongly continous semigroup of contractions on
$B_b(E)$ generated by $L - \frac1\e V_\e$. 

For functions $f\in B_{\rm b}(\OO)$ we define
\[\wt f(x) := \begin{cases} f(x), & x\in\OO, \\ 0, & x\in\complement{\OO}.\end{cases}\] 

\bpr[cf. \hbox{\cite[Proposition 3.5]{DL}}]\label{prop:pointwise}
 For all $f\in B_{\rm b}(\OO)$, $x\in \OO$, and $t\ge 0$,
 \[\lim_{\e\downarrow 0}P_\e(t)\wt f(x)=P_\OO (t)f (x).\]
\epr

\bpf
For $t=0$ the result is trivial, so we may assume that $t>0$. 
Fix $x\in \OO$.

On the set $\{\tau^x_{{\OO}} > t\}$ we have 
$X^x(s)\in {\OO} $ for all $s\in [0,t]$ and therefore $V_\e(X^x(s))=0$
for all $s\in [0,t]$ provided $\e>0$ is sufficiently small. 
If, on the other hand, $\o\in \{\tau^x_{{\OO}} \le t\}$, then by path continuity 
 we have
$X^x(t_0(\omega),\omega)\in \partial \OO$ for some $t_0(\o)\in (0,t]$, 
and therefore $V_\e(X^x(t_0(\omega),\omega)) = 1$ for all $\e>0$.
Hence for some small enough $\delta(\omega)>0$ we have $V_\e(X^x(s,\omega)) \ge \frac12$ for all 
$s\in [t_0(\omega)-\delta(\omega), t_0(\o)]$.
 Then
\[\int_0^tV_\e(X^x(s,\omega))  \,ds\ge \int_{t_0(\omega)-\d(\omega)}^{t_0(\omega)}V_\e(X^x(s,\omega)) \,ds  
\ge \tfrac12\d(\omega) >0,\]
and therefore
$\limsup_{\e\downarrow 0} e^{-\frac1\e\int_0^tV_\e (X^x(s,\omega)) \,ds}\le\lim_{\e\downarrow 0}
e^{-\frac{\d(\omega)}{2\e}} = 0.$

Using these facts, by dominated convergence we obtain
\begin{align*} \lim_{\e\downarrow 0} P_\e(t)\wt f(x)
& = \E[f(X^x(t))\one_{\{\tau^x_{\OO} > t\}}] 
 + \lim_{\e\downarrow 0}\int_{\{ \tau^x_{\OO} \le t\}}\wt f(X^x(t))
e^{-\frac1 \e\int_0^tV_\e(X^x(s)) \,ds} \,d\P
\\ & = P_\OO(t)f(x)+ \lim_{\e\downarrow 0}\int_{\{ \tau^x_{\OO} \le t\}}\wt f(X^x(t))
e^{-\frac1 \e\int_0^tV_\e(X^x(s)) \,ds} \,d\P
\\ & = P_\OO(t)f(x).
\end{align*}
\epf

\begin{remark} \label{rem:comparison}
The papers \cite{DL} considers closed domains $K$ are used instead of open sets $\OO$.
This has the advantage that one can work with one potential $V$ which vanishes on $K$ and is strictly
positive outside $K$.
In this setting, however, we don't see how to prove the analogue Proposition \ref{prop:pointwise} without
any assumptions on the boundary of $K$ (the problem being the identity $P_K(0)f(x) = f(x)$ for points $x\in \partial K$,
which in general need not hold).
\end{remark}

\begin{proposition}\label{prop:PKcontr}
The semigroup $P_\OO $ has a unique extension to a $C_0$-semigroup of contractions on $\LK$.
\end{proposition}
\bpf
First we prove that each of the operators $P_\OO(t)$ extends uniquely to a contraction on $\LK$.
By the Cauchy-Schwarz inequality, for all $f\in B_{\rm b}(\OO)$ and $x\in \OO$ we have
\begin{align*}
(P_\OO (t)f(x))^2&=(\E[\wt f(X^x(t))\one_{\{\tau^x_{{\OO}}> t  \}}])^2\\
&\le\E[\wt f^2(X^x(t))\one_{\{\tau^x_{{\OO}}> t  \}}]\le P(t)\wt{f}^2(x).
\end{align*}
Hence,
\[\int_\OO (P_\OO (t)f)^2\,d\mui\le \int_EP(t)\wt{f}^2\,d\mui=
\int_E\wt{f}^2\,d\mui= \int_\OO f^2\,d\mui.\]
This proves the asserted contractivity.

To prove strong continuity on $L^2(\OO,\mui)$, 
first let $f\in B_{\rm b}(\OO)$.
Then, by the path continuity of $X^x$,  for all $x\in\OO$  
we have $\lim_{t\downarrow 0} X^x(t) = x$ and $\tau_\OO^x >0$, and therefore
$$\lim_{t\downarrow 0} P_\OO (t)f(x)= \lim_{t\downarrow 0} \E[ f(X^x(t))\one_{\{\tau^x_{{\OO}} > t\}}] = f(x)$$
by dominated convergence. 
Again by dominated convergence, this implies that $\lim_{t\downarrow 0}  P_\OO (t)f = f$ in $L^2(\OO,\mui)$.
For general $f\in \LK$, strong continuity in $L^2(\OO,\mui)$ follows by density.
\epf

From now on, $P_\OO $ always denotes the $C_0$-semigroup of contractions
on $L^2(E,\mui)$ whose existence is assured by the proposition. We denote by $L_\OO $ its
generator.

\bpr[cf. \hbox{\cite[Proposition 3.7]{DL}}]\label{appL2}
For all $f\in \LK$ and $t>0$,
\beq\label{AS}
\lim_{\e\downarrow 0} \, (P_\e(t)\wt{f})|_\OO =P_\OO (t)f\quad {\rm in } \ \LK.
\eeq
Moreover, for all $\la>0$ with $\la\in\varrho(L-\frac 1 \e V_\e )$ we have $\la\in\varrho(L_\OO )$ and
 \beq\label{AR}
 \lim_{\e\downarrow 0}(R(\la,L-\frac1 \e V_\e )\wt f)|_\OO =R(\la,L_\OO )f\quad  {\rm in }\
\LK.
\eeq
\epr
Here, for an operator $A$ and $\la\in\varrho(A)$, $R(\la,A):= (\la-A)^{-1}$ denotes
the associated resolvent
operator.

\bpf
First let $f\in C_{\rm b}(\OO)$. Then for all $x\in \OO$ we have the pointwise bounds
$$
|P_\e(t)\wt f(x)| = |\E [\wt f(X^x(t))e^{-\frac1\e\int_0^tV_\e (X^x(s)) \,ds}]|
\le\n f\n_\ii
$$
and
$$
|P_\OO(t)f(x)|= |\E[f(X^x(t))\one_{\{ \tau^x_{{\OO}} > t\}}]|\le\n f\n_\ii.
$$
Hence by Proposition \ref{prop:pointwise} and dominated convergence theorem we obtain
$$\lim_{\e\downarrow 0} \n (P_\e(t)\wt f)|_\OO-P_\OO(t)f\n_{\LK} = 0$$ for all $f\in C_{\rm b}(E)$.
Since $P_\e$ and $P_\OO$ are contractive in $L^2(E,\mui)$ and $L^2(\OO,\mui)$, respectively,
this convergence extends to arbitrary $f\in L^2(\OO,\mui)$.

Finally, (\ref{AR}) follows from (\ref{AS}) by taking Laplace transforms.
\epf

Recalling the definition ${W}_H^{1,2}(\OO,\mui)  := \D(D^H)$, we now define
$$\mathring{W}_H^{1,2}(\OO,\mui) := \big\{f\in \LK: \ \wt f\in \D(D^H), \ D^H \wt f =0 \ 
\mui\hbox{-a.e. on }\complement \OO\big\}.$$
 Thus, $\mathring{W}_H^{1,2}(\OO,\mui)$ is the natural domain
of the part of $D^H$ in $L^2(\OO,\mui)$. We shall study this operator in more detail in the next section.

\bth[cf. \hbox{\cite[Theorem 3.8]{DL}}]
For all $\la>0$ and $f\in \LK$ we have $\phi:=R(\la,L_\OO )f \in \mathring{W}_H^{1,2}(\OO,\mui) $ and
\beq\label{WS}
\la\int_\OO \phi v\,d\mui+\int_\OO [BD ^H\phi,D ^Hv]\,d\mui=\int_\OO  fv\,d\mui\quad
\forall v\in \mathring{W}_H^{1,2}(\OO,\mui) .
\eeq
\eth

\bpf
Fix $\la>0$ and $f\in \LK$.
For $\e>0$ set $$\phi_\e:= R(\la,M_\e)\wt f =R(\la,L-\frac1\e V_\e )\wt{f}.$$
Then $\phi_\e \in \D(M_\e) = \D(L)$, so $\phi_\e \in \D(D^H) = W_H^{1,2}(E,\mui)$, and
by (\ref{R}) and (\ref{DR}) (applied to the potentials $V_\e$) we obtain
\[ \n \phi_\e\n^2_{W_H^{1,2}(E,\mui)}=\n \phi_\e\n^2_{\L}+\n D^H
\phi_\e\n^2_{\LH}\le \big(\frac{1}{\la^2}+\frac{2}{\la}\big)\n f\n^2_{\LK}.\]
Therefore there exists a sequence
$\e_j\to 0$ and a function $\psi\in W_H^{1,2}(E,\mui)$ such that 
$\phi_{\e_j}\ra\psi$ weakly in $W_H^{1,2}(E,\mui)$ as $j\ra\ii$. Let us prove
that $\psi=\wt{\phi}$.

For every $g\in \LK$, by (\ref{AR}) we have
\[\int_\OO\psi g\,d\mui
=\lim_{j\to\infty}
\int_\OO\phi_{\e_j}g\,d\mui
=\int_\OO\phi g\,d\mui.\]
Thus $\psi|_\OO=\phi$. Next we want to prove that $\psi|_{\complement \OO}=0$. The weak
convergence $\phi_{\e_j}\to \psi$ in $W_H^{1,2}(E,\mui)$ implies weak convergence in $L^2(E,\mui)$
and hence in $L^2(\complement\OO,\mui)$. Recalling 
that $V_\e \equiv 1$ on $\complement \OO$, we obtain 
\[\int_{\complement \OO}\psi^2 \,d\mui=\lim_{j\to\infty}
\int_{\complement \OO}\phi_{\e_j}\psi   \,d\mui =\lim_{j\to\infty}
\int_{\complement \OO}\phi_{\e_j}\psi V_{\e_j}  \,d\mui.\]
Using (\ref{VR}),
\begin{align*}
\Big|\int_{\complement \OO}\phi_{\e_j}\psi
V_{\e_j}  \,d\mui\Big|&\le\Big(\int_{\complement \OO}|\phi_{\e_j}|^2 V_{\e_j} 
\,d\mui\Big)^{1/2}\Big(\int_{\complement \OO}|\psi|^2 V_{\e_j}  \,d\mui\Big)^{1/2}\\
&\le\Big(\frac{\e_j}{\la}\int_E|\wt{f}|^2\,d\mui\Big)^{1/2}\Big(\int_{\complement \OO}|\psi|^2
V_{\e_j}  \,d\mui\Big)^{1/2}.
\end{align*}
Upon letting $j\ra\infty$, we obtain that $\psi|_{\complement \OO}=0$ $\mui$-almost everywhere.

By what has been proved so far,
$\phi_{\e_j}\ra\wt{\phi}$ weakly in $W_H^{1,2}(E,\mui)$. 

Next we will prove that $(D^H \wt \phi)|_{\complement \OO}=0$ $\mui$-almost everywhere.
By \eqref{DR}, the functions  $D^H \phi_\e$ are uniformly bounded in $L^2(E,\mui)$, and therefore
there exists a (possibly different) sequence
$\e_j\to 0$ and a function $\xi\in W_H^{1,2}(E,\mui)$ such that 
$D^H\phi_{\e_j}\ra\xi$ weakly in $L^2(E,\mui)$ as $j\ra\ii$. 
Then, arguing as before,
\[\int_{\complement \OO}\xi^2\,d\mui=\lim_{j\to\infty}
\int_{\complement \OO}D^H\phi_{\e_j}\xi V_{\e_j}  \,d\mui.\]
Using \eqref{DV},
\begin{align*}
\Big|\int_{\complement \OO}D^H\phi_{\e_j}\xi
V_{\e_j}  \,d\mui\Big|&\le\Big(\int_{\complement \OO}\n D^H\phi_{\e_j}\n_H^2V_{\e_j} 
\,d\mui\Big)^{1/2}\Big(\int_{\complement \OO}|\xi|^2 V_{\e_j}  \,d\mui\Big)^{1/2}\\
&\le\Big(\frac{\e_j^{1/2}}{\la^{1/2}}\int_E|\wt{f}|^2\,d\mui\Big)^{1/2}\Big(\int_{\complement \OO}|\xi|^2
\,d\mui\Big)^{1/2}.
\end{align*}
Upon letting $j\to \infty$, we obtain that $\xi|_{\complement \OO}=0$ $\mui$-almost everywhere.
Moreover, the closedness (and hence, by the Hahn-Banach theorem, weak closedness) of $D^H$
gives $ D^H \wt \phi = \xi$. This proves 
that $(D^H \wt \phi)|_{\complement \OO}=0$ $\mui$-almost everywhere.

Combining what we have proved so far, we see that $\phi\in
\mathring{W}_H^{1,2}(\OO,\mui).$ 
Next we multiply the identity $\la \phi_{\e_j} - L\phi_{\e_j} +\frac{1}{\e_j}V_{\e_j} \phi_{\e_j}=\wt{f}$ with
$\wt v$, where
$v\in \mathring{W}_H^{1,2}(\OO,\mui) $. Upon integrating over $E \setminus(\OO\setminus \OO_\e)$
and noting that $V_{\e_j}  \wt{v} \equiv 0$ on this set,  we obtain
\begin{align*}
\ & \int_{E \setminus(\OO\setminus \OO_\e)}(\la-L)\phi_{\e_j}\wt{v}\,d\mui
\\ & \ \  = \int_{E \setminus(\OO\setminus \OO_\e)}(\la-L)\phi_{\e_j}\wt{v}\,d\mui
+\frac1{\e_j} \int_{E\setminus(\OO\setminus \OO_\e)}V_{\e_j}\phi_{\e_j}\wt{v}  \,d\mui
=\int_{E \setminus(\OO\setminus \OO_\e)}\wt{f}\wt{v}\,d\mui.\end{align*}
Passing to the limit for $j\ra\infty$ and using Proposition \ref{prop:div}, we obtain 
\begin{align*}
 \la\int_{E}\phi_{\e_j}\wt{v}\,d\mui+\int_{E}[BD^H\phi_{\e_j}, D^H\wt{v}]\,d\mui
& = \int_{E}(\la-L)\phi_{\e_j}\wt{v}\,d\mui
 = \int_{E}\wt{f}\wt{v}\,d\mui.
\end{align*}
This proves (\ref{WS}).\epf

It follows from this theorem that $\D(L_\OO)\subseteq \mathring{W}_H^{1,2}(\OO,\mui)$.
 In particular,
the space $\mathring{W}_H^{1,2}(\OO,\mui)$ is dense in $\LK$.

Consider the bilinear form (recall that we work over the real scalars)
$$(f,g) \mapsto \int_\OO [BD^H f, D^H g]\,d\mui, \quad f,g \in \mathring{W}_H^{1,2}(\OO,\mui).$$
It is an easy consequence of the identity  $[Bh,h] =\frac1 2\n h\n_H^2$ 
(see Proposition \ref{prop:MN})
that this form is densely defined, continuous, accretive, and closed.
Arguing as in \cite[Proposition 4.3]{MN2} we see that it is in fact sectorial, and therefore 
we can define a closed densely defined operator  $-M_\OO$, which we will call the {\em 
Dirichlet Ornstein-Uhlenbeck
operator}, on $L^2(\OO,\mui)$ with this form in the usual way (see \cite[Section 1.2.3]{O}), and  
$M_\OO$ generates a strongly continuous analytic contraction semigroup
on $\LK$.

\begin{theorem} 
We have $L_\OO = M_\OO$. As a consequence, 
the semigroup $P_\OO$ is a strongly continuous
analytic contraction semigroup on $\LK$.
\end{theorem}
\begin{proof}
Using the notation of the previous proposition, from (\ref{WS}) it follows that
if $f\in \LK$ and $\lambda>0$, then $
\phi = R(\lambda,L_\OO)f\in \D(M_\OO)$ and
$$ \la \phi - M_\OO \phi = f = \la \phi - L_\OO \phi.$$
It follows that $\D(L_\OO) \subseteq \D(M_\OO)$ and that
$L_\OO= M_\OO$ on $\D(L_\OO)$. Since both operators are semigroup generators,
this implies that $\D(L_\OO) = \D(M_\OO)$ and
$L_\OO=M_\OO$.
\end{proof}

We conclude this section with a gradient
estimate for non-symmetric Ornstein-Uhlenbeck semigroups.
Da Prato and Lunardi studied the symmetric case (see \cite[Section 3.3, consequence (iii), and Proposition 3.9]{DL}).

\begin{theorem}[Gradient estimates] 
For all $f\in \L$
\[\n D^HP_\e(t)f\n_{\LH}\le \frac{C}{\sqrt{t}}\n f\n_{\L},\]
and for all $f\in \LK$
\[\n D ^HP_\OO (t)f\n_{\LKH}\le \frac{C}{\sqrt{t}}\n f\n_{\LK}.\]
\end{theorem}
\begin{proof}
Using (\ref{DR}) and setting $t=\frac 1 \la$ we observe that, for all $g\in L^2(E,\mu_\infty)$,
\[\n D^H(I-tL+\frac t \e V_\e )^{-1}g\n_{\LH}\le \sqrt{\frac{2}{t}}\n g\n_{\L}.\]
Then using this estimate with the $L^2$-contractivity of $P_\e(t)$ and its $L^2$-analyticity we obtain
\[\n D^HP_\e(t)f\n_{\LH}\le \sqrt{\frac{2}{t}}\n (I-tL+\frac t \e V_\e )P_{\e}(t)f\n_{\L}
\le \frac{C_\e}{\sqrt{t}}\n f\n_{\L},\]
with a constant $C_\e$ which, as an inspection of the proof shows, 
can be uniformly bounded from above independently of $\e>0$.
Applying the method of proof of the inequality (\ref{DR}) on the identity (\ref{WS}) yields
\begin{eqnarray}\label{GR}
\n D ^HR(\la,L_\OO )f\n_{\LKH}\le \sqrt{\frac{2}{\la}}\n f\n_{\LK}.
\end{eqnarray}
Then arguing as above we obtain
\[\n D ^HP_\OO (t)f\n_{\LKH}\le \frac{C}{\sqrt{t}}\n f\n_{\LK}.\]
\end{proof}

\section{Boundedness of the Riesz transform for $L_\OO $}
\label{sec:RT}

In this section we obtain sufficient conditions for the boundedness on $\LK$
of the Riesz transform associated with $L_\OO $.
Observe that when $L_\OO $ is selfadjoint (i.e. when $B=\frac1 2 I$), this follows from the identities
\begin{align*}
\n (-L_\OO) ^{1/2}f\n^2_{\LK} & = -\int_\OO  L_\OO f\cdot f \,d\mui \\ & =\frac12
\int_\OO [D_\OO ^Hf,D_\OO ^Hf]\,d\mui=\frac12\n D_\OO ^Hf\n^2_{\LKH}.
\end{align*}
In order to discuss the non-selfadjoint case we need to introduce some auxiliary operators.

We begin by defining the operator $D_\OO ^H$ 
with domain $\D(D_\OO ^H) : = \mathring{W}_H^{1,2}(\OO,\mui)$
by
$$ D_\OO ^H f := D^H \wt f, \quad f\in \D(D_\OO ^H).$$
By the definition of $\mathring{W}_H^{1,2}(\OO,\mui)$, 
 $ D^H \wt f$ vanishes $\mui$-almost everywhere on $\complement \OO$,
so that it can indeed be identified with an element of $L^2(\OO,\mui;H)$.

\begin{lemma}
The operator $D_\OO ^H$ is closed and densely defined in $L^2(\OO,\mui)$.
\end{lemma}
\begin{proof}
We have already seen that $\D(D_\OO^H) = \mathring{W}_H^{1,2}(\OO,\mui)$ is dense in $\LK$.
To see that $D_\OO^H$ is closed, let $f_n\in \D(D_\OO^H)$ be such that 
$f_n\to f$ in $L^2(\OO,\mui)$ and $D_\OO^H f_n \to g$  in $L^2(\OO,\mui;H)$. Then $\wt f_n \to \wt f$ in $L^2(E,\mui)$ and 
$D^H \wt f_n \to \wt g$ in $L^2(E,\mui;H)$, so $\wt f\in \D(D^H)$ and $D^H \wt f = \wt g$.
But this is the same as saying that $f\in \D(D_\OO^H)$ and $D_\OO ^H f = g$.
\end{proof}

Thanks to this lemma, the adjoint operator ${D_\OO ^H}\s = ({D_\OO ^H})\s$ is well-defined as a 
closed densely defined operator on $\LKH$.

The next lemma is a straightforward consequence of the definition of $L_\OO$ in terms of the bilinear form $l_\OO$.

\begin{lemma}\label{lem:L} We have
$$\D(L_\OO) = \{f\in \D(D_\OO^H): BD_\OO^H f\in \D(D_\OO^{H*})\}
= \{f\in\D(D_\OO^H): D_\OO^H f\in\D(D_\OO^{H*}B)\},$$
and for all $f\in\D(L_\OO)$ we have 
$$ L_\OO f = - D_\OO^{H^*}(BD_\OO^H)f =- (D_\OO^{H*}B)D_\OO^H f.$$
\end{lemma}
 
Consider the form
$$\underline{l}_\OO^I (F,G):= \int_\OO [{D_\OO ^H}\s F,{D_\OO ^H}\s G]_H\,d\mu_\infty$$ for
$F,G\in \D(\underline{l}_\OO^I ):=\D({D_\OO ^H}\s)$.
This form is accretive, densely defined and closed, and since it
is symmetric, it is sectorial.
Therefore the associated
operator, which we denote by $D_\OO ^H{D_\OO ^H}\s$, is densely defined, closed, and selfadjoint, with
domain
\[\D(D_\OO ^H{D_\OO ^H}\s) =\{F\in \D({D_\OO ^H}\s): \ {D_\OO ^H}\s F\in \D(D_\OO ^H)\}.\]
Since $B$ is bounded and coercive we have equivalences of norms
$$
\n Bu\n\eqsim\n u\n\eqsim\n B\s u\n.
$$
As a consequence, $B$ is
boundedly invertible. By the argument of \cite[Proposition 5.1]{MN2}),
it follows from \cite[Proposition 7.1]{AMN} that
 the operator $$\underline{L}_\OO :=-D_\OO ^H{D_\OO ^H}\s B$$
with domain
$$ \D(D_\OO ^H{D_\OO ^H}\s B) = \{ F\in L^2(\OO;H): \ BF\in  \D(D_\OO ^H{D_\OO ^H}\s)\}$$
 is closed, densely defined,
and sectorial.
In particular, $\underline{L}_\OO $ generates a bounded analytic semigroup,
denoted by $\underline{P}_\OO (t)$, on $\LKH$ (see \cite[Theorem 4.6]{EN}).

\begin{lemma}\label{lem:commute} For all $g\in \D(L_\OO )$ and $t>0$ we have 
$(I-tL_\OO )^{-1}g\in \D(D_\OO ^H)$ and 
\beq\label{1}
 D_\OO ^H(I-tL_\OO )^{-1}g=(I-t\underline{L}_\OO )^{-1}D_\OO ^Hg
\eeq
 and
 \beq\label{2}
  (I-tL_\OO )^{-1}{D_\OO ^H}\s B D_\OO ^H g={D_\OO ^H}\s B(I-t\underline{L}_\OO )^{-1}D_\OO ^H g.
 \eeq
\end{lemma}
\begin{proof}
The set $\calA:=\{f\in \D(L_\OO ): \  L_\OO f\in \D(D_\OO ^H)\}$ is dense
(it contains the dense set $\calB = \{R(\lambda,L_\OO)g:\ \lambda>0, \ g\in\D(L_\OO)\}$)  
and invariant under $P_\OO (t)$,
and therefore it is a core for $\D(L_\OO )$.

For  all $f\in \calA$ we have, using Lemma 
\ref{lem:L} to justify the formal computation, 
\[D_\OO ^HL_\OO f=-D_\OO ^H{D_\OO ^H}\s BD_\OO ^Hf=\underline{L}_\OO D_\OO ^Hf.\]
Multiplying the resulting identity
\[D_\OO ^H(I-t L_\OO )f=(I-t \underline{L}_\OO )D_\OO ^Hf\]
on the left by $(I-t \underline{L}_\OO )^{-1}$ and taking $f = (I-t {L}_\OO )^{-1} g$
with $g\in \D(L_\OO )$ (in which case we have $f\in \calA$),
the identity in (\ref{1})
is obtained for functions  $g\in \D(L_\OO )$.

Next, $\D(L_\OO)$ is a core for $\D(D_\OO^H)$ and by \eqref{1},
for all $f\in \D(D_\OO ^H)$ we have
 \begin{align*}
[{D_\OO ^H}f, B(I-t\underline{L}_\OO )^{-1}D_\OO ^H g]
& = [{D_\OO ^H}f, BD_\OO ^H (I-t{L}_\OO )^{-1} g]
\\ & = -[f, L_\OO  (I-t{L}_\OO )^{-1} g]
 =-[f,(I-t{L}_\OO )^{-1} L_\OO  g]
\\ & = [f,(I-t{L}_\OO )^{-1}{D_\OO ^H}\s BD_\OO ^H g].
\end{align*}
This show that
$B(I-t\underline{L}_\OO )^{-1}D_\OO ^H g$ is in $\D({D_\OO ^H}\s)$ and \eqref{2} holds.
\end{proof}

By standard semigroup theory, the above lemma implies the identity
$$ \underline{P}_\OO(t)D_\OO^H f = D_\OO^H P_\OO(t)f,$$
first for $f\in \D(L_\OO)$ and then for $f\in\D(D_\OO^H)$, using that 
$\D(L_\OO)$ is a core for $\D(D_\OO^H)$. In particular, we see that
the semigroup $\underline{P}_\OO$ maps $\overline{\Ran(D_\OO^H)}$ into itself.
From now on, we shall always consider $\underline{P}_\OO$ as a semigroup 
on this space. By a slight abuse of notation its generator, which is the part of $\underline{L}_\OO$
in $\overline{\Ran(D_\OO^H)}$, will be denoted again by $\underline{L}_\OO$. 

On the product space $\L \oplus \ovRD$ we now consider the operator
 \[\Pi_\OO :=\left(\begin{matrix}
   0 & {D_\OO ^H}\s B \\
   D_\OO ^H & 0 \\
   \end{matrix}
         \right)
\]
with domain $\D(\Pi_\OO) = \D(D_\OO^H)\oplus \D(D_\OO^{H*}B)$, where, by the 
same abuse of notation, we denote by $\D_\OO^{H^*}B$ the domain of the 
part of $D_\OO^{H*}B$ in $\ovRD$.
Observe that
 \[\Pi^2_\OO :=\left(
   \begin{matrix}
   -L_\OO  & 0 \\
   0 & -\underline{L}_\OO  \\
   \end{matrix}
         \right).\]

A densely defined closed linear operator $A$ is called {\em bisectorial} 
if $i\R\setminus\{0\}\subseteq \varrho(A)$ and
$$ \sup_{t\in\R\setminus 0} \n (I-itA)^{-1} \n < \infty.$$
A standard Taylor expansion argument implies that there exists an $\theta\in (0,\frac12\pi)$
such that the open bisector of angle $\theta$ around the imaginary axis belongs to
$\varrho(A)$ and the above uniform boundedness estimate extends to this bisector.
 
Let us recall the following result (see \cite[Section (H)]{ADM}) which uses McIntosh's notion of a bounded
functional calculus. Let $\theta\in (0,\frac12\pi)$ be given. A 
sectorial operator $T$ on a Banach space $F$ admits a {\em bounded functional calculus of angle $\theta$} 
if the Dunford functional calculus
of $T$ extends to a bounded homomorphism 
$$ H^\infty(\Sigma_\theta) \to \calL(F), \qquad f\mapsto f(T).$$
Here $\Sigma_\theta = \{z\in \C\setminus\{0\}: \ |\arg(z)|<\theta\}$ is the open sector
in the complex right half-plane with aperture $\theta$. 
For a detailed treatment we refer the reader to \cite{ADM, Haase, KuW}.
The bounded functional calculus 
$$ H^\infty(\Sigma_\theta\cup -\Sigma_\theta) \to \calL(F), \qquad f\mapsto f(T)$$
for bisectorial operators $T$ on $E$ is defined similarly.

\bpr\label{BHFC}
If $\Pi$ is a bisectorial operator on a Hilbert space $\mathscr{H}$, then $\Pi^2$ is sectorial on $\mathscr{H}$
and for each $\theta\in (0,\frac\pi 2 )$ the following assertions are equivalent:
\ben
\item $\Pi$ admits a bounded functional calculus on a bisector of angle $\theta$;
\item $\Pi^2$ admits a bounded functional calculus on a sector of angle $2\theta$.
\een
\epr

Now we are ready to state and prove the first main result of this section.
Examples where the conditions of the theorem are fulfilled are given subsequently.

\bth\label{thm:Riesz} Suppose that $-\underline{L}_\OO $ admits a bounded holomorphic functional calculus
on $\ovRD$.
Then,
\begin{equation}\label{RT1}
\begin{aligned}
\D(D_\OO ^H) & = \D((-L_\OO )^{1/2}), \\
\D({D_\OO ^H}\s B) & = \D((-\underline L_\OO )^{1/2}),
\end{aligned}
\end{equation}
with equivalence of the homogeneous seminorms
\begin{equation}\label{RT2}
\begin{aligned}
\n D_\OO ^H f\n_{\LKH} & \eqsim\n (-L_\OO )^{1/2}f||_{\LK}, \\
 \n {D_\OO ^H}\s B g\n_{\LK} & \eqsim\n (-\underline{L}_\OO )^{1/2}g||_{\LKH}.
\end{aligned}
\end{equation}
\eth
\bpf
We shall prove that $\Pi_\OO $ is
bisectorial on $L^2(\OO,\mui)\oplus \ovRD$. 
Assuming this for the moment, we first show how the result follows from this.

Since $-L_\OO $ and $-\underline{L}_\OO $ have 
bounded functional calculi on suitable sectors of angle $<\frac12\pi$
(for $-L_\OO$ this follows from the fact that $L_\OO$ generates an analytic contraction semigroup),
the same is true for $\Pi_\OO ^2$ and hence, by
Proposition \ref{BHFC}, $\Pi_\OO $ has a bounded functional calculus on a bisector of angle $<\frac14\pi$. This implies
the boundedness of the operators
${\Pi_\OO }/{\sqrt{\Pi_\OO ^2}}$ and of ${\sqrt{\Pi_\OO ^2}}/{\Pi_\OO }$
(apply the functional calculus of $\Pi_\OO $ to the
the bounded holomorphic functions $z/\sqrt{z^2}$ and $\sqrt{z^2}/z$).
By a standard argument, this implies (\ref{RT1}) and (\ref{RT2}); we refer to \cite{AKM, MN2}
for the details.

It remains to prove the bisectoriality of $\Pi_\OO $.
Fix $t\in \R\setminus\{0\}$ and consider the operator matrix
 \[R_t:=\left(
   \begin{matrix}
     (I-t^2L_\OO )^{-1} & it(I-t^2L_\OO )^{-1}{D_\OO ^H}\s B \\
     itD_\OO ^H(I-t^2L_\OO )^{-1} & (I-t^2\underline{L}_\OO )^{-1} \\
   \end{matrix}
 \right)
 \]
By Lemma \ref{lem:commute}, the identity $R_t (I-it\Pi_\OO )=I$ holds on the linear subspace
of all $(g,G) \in \LK\oplus \ovRD$
with $g\in \D(L_\OO )$ and $G = D_\OO^H g'$ with $g'\in \D(L_\OO )$. Since $\D(L_\OO )$ is a core
for $\D(D_\OO ^H)$, this linear subspace is dense and the
identity extends to all pairs $(g,G) \in \LK\oplus \ovRD$.

This shows that $R_t$ equals the resolvent $(I-it\Pi_\OO )^{-1}$ defined on
$\LK\oplus\ovRD$.
 Let us now study the boundedness of each of the entries of the matrix $R_t$.

We have already seen that
  \[\n(I-t^2L_\OO )^{-1}\n_{\LK}\le 1\ \ {\rm and }\ \
\n(I-t^2\underline{L}_\OO )^{-1}\n_{\LKH}\le C,\]
with a constant $C$ independent of $t\in \R\setminus\{0\}$.

Taking $\lambda=\frac{1}{t^2}$ in (\ref{GR}) we obtain
\beq\label{NDR}
\n tD_\OO ^H(I-t^2L_\OO )^{-1}\n_{\LKH}\le 2.
\eeq
The previous argument hold also if we replace $B$ by $B\s$ in the definition of $L_\OO $, so we
obtain (\ref{NDR}) with $L_\OO \s$ instead of $L_\OO $.
Then, using Lemma \ref{lem:L} to see that $L_\OO = (D_\OO^{H*}B)D_\OO^H$ implies
$L_\OO\s = D_\OO^{H*}(D_\OO^{H*}B)\s = D_\OO^{H*}(B\s D_\OO^H) = (D_\OO^{H*}B\s )D_\OO^H$,
and using \eqref{NDR} with $L_\OO \s$ instead of $L_\OO $, we obtain
 \begin{align*} \n t B\s D_\OO ^H(I-t^2L\s_\OO )^{-1}\n_{\LKH}
& \le  \n B\n \n t
D_\OO ^H(I-t^2L\s_\OO )^{-1}\n_{\LKH} 
\\ & \le 2\n B\n \le 1,
\end{align*}
and by duality we obtain
\[\n t(I-t^2L_\OO )^{-1}{D_\OO ^H}\s B \n_{\LK}\le 1. \]

As a consequence, the operators $(I-it\Pi_\OO )^{-1}$ are uniformly bounded on the space $\LK\oplus
\ovRD$ for all $t\in \R\setminus\{0\}$. By standard arguments, this implies 
that $\Pi_\OO $ is bisectorial on $\LK\oplus
\ovRD$.
\epf

 The condition that $-\underline{L}_\OO $ has a functional calculus is satisfied when $L_\OO $ is selfadjoint.
Indeed, then $\underline{L}_\OO $ is selfadjoint as well and
$-\underline{L}_\OO $, being non-negative and selfadjoint, admits a bounded holomorphic
calculus.

\medskip\noindent
{\em Open problem.} If $-\underline{L}$ admits a bounded holomorphic
calculus on $\overline{\Ran(D^H)}$, does  $-\underline{L}_\OO $ admit a bounded holomorphic
calculus on $\ovRD$?

\medskip
An affirmative answer would imply that the condition of Theorem \ref{thm:Riesz} is always
satisfied in case $H = E = \R^n$ (as is explained in the discussion below \cite[Theorem 2.2]{MN2}).

In order to state a second open problem we need to introduce some notation.
We begin with a lemma
which asserts that we can define $\underline{D}_\OO ^H = D_\OO ^H\otimes I$
as a closed and densely defined operator from
$L^2(\OO,\mui)$ to $L^2(\OO,\mu_\infty;H\wh\otimes H)$.
Here, and in what follows, we denote by $\otimes $ and $\wh \otimes $ the
algebraic tensor product and the completed Hilbert space tensor product respectively.

\begin{lemma}
The mapping $\underline{D}_\OO^H: \mathring{W}_H^{1,2}(\OO,\mui)\otimes H \to L^2(\OO,\mui; H\wh \otimes H)$
defined by $$ \underline{D}_\OO^H (f\otimes h) := D_\OO^H f \otimes h $$
is closable as a operator from $ L^2(\OO,\mui;H)$ to $L^2(\OO,\mui; H\wh\otimes H)$. 
\end{lemma}
\begin{proof}
 Suppose $F_n\to 0$ in $L^2(\OO,\mui;H)$, with each $F_n \in \mathring{W}_H^{1,2}(\OO,\mui)\otimes H$,
and  $\underline{D}_\OO^H F_n \to G$ in $L^2(\OO, \mui; H\wh\otimes H)$.
Then, for each $h\in H$, $[F_n,h] \to 0 $ in $L^2(\OO,\mui)$, $[F_n,h] \in \mathring{W}_H^{1,2}(\OO,\mui)$,
and  $$[\underline{D}_\OO^H F_n,h] = D_\OO^H [F_n,h]  \to [G,h]$$ in $L^2(\OO,\mui; H)$.
The closedness of $D_\OO^H$ implies $[G,h]=0$ in $L^2(\OO,\mui; H)$ for all $h\in H$, and therefore
$G=0$ in $L^2(\OO,\mui; H\wh\otimes H)$.
\end{proof}

By a slight abuse of notation, from now one we shall denote by $\underline{D}_\OO^H$ 
the closure of this operator and by
$\D(\underline{D}_\OO^H)$ its domain.
The closed operator $\underline{D}^H$ with domain $\D(\underline{D}^H)$ is defined
similarly (see, e.g., \cite[Section 11]{MN2}). 

Let $$
\begin{aligned}
\mathring{W}_H^{2,2}(\OO,\mui) 
& = \big\{f\in \D(D_\OO^H): \ D_\OO^H f \in  \D(\underline D_\OO^H)\big\} \\ 
& =
\big\{f\in L^2(\OO,\mui): \ \wt f \in \D(D^H), \ D^H f \in  \D(\underline D^H), \\
& \qquad D^H \wt f = 0 \ \mui\hbox{-a.e. on }\complement \OO,  \
       \underline{D}^H (D^H \wt f) = 0 \ \mui\hbox{-a.e. on }\complement \OO\big\}. 
\end{aligned}
$$ 
With respect to the norm $\n \cdot\n_{\mathring{W}_H^{2,2}(\OO,\mui) }$ defined by
$$ \n f\n_{\mathring{W}_H^{2,2}(\OO,\mui)}^2 = \n f\n^2 + \n D_\OO^H f\n^2 + \n \underline{D}_\OO^H (D_\OO^H f)\n^2,$$
this space is a Hilbert space. 

\medskip\noindent
{\em Open problem.} Under what conditions on $A$ and $\OO$ do we have a continuous inclusion 
\begin{equation*} \D(L_\OO ) \subseteq \mathring{W}_H^{2,2}(\OO,\mui) ?
\end{equation*}
For $\OO = E$, this inclusion is obtained in \cite[Proposition 11.1(iii)]{MN2}.
In the case of non-trivial open domains $\OO\subset E$, Da Prato and Lunardi \cite{DL2}
obtained the inclusion in the case $A = I$ under suitable regularity conditions
on $\partial \OO$ and showed that the inclusion may fail if no such conditions are imposed.  
The methods of \cite{MN2} seem not to adapt very well to the domain setting.
\medskip

We conclude with a Poincar\'e inequality for $L_\OO $.

\bth[Poincar\'e inequality for $L_\OO $]\label{thm:PoincII} 
Suppose that $\mui(\complement \OO)>0$ and that
 $S_\infty$ is uniformly exponentially stable. If
$-\underline{L}_\OO $ admits a bounded holomorphic functional calculus
on $\ovRD$, there is a constant $C$ such that for all $u \in \mathring{W}^{1,2}(\OO,\mui)$
we have
\[\n u \n_{\LK}\le C\n D_\OO ^Hu\n_{\LKH} .\]
\eth

\bpf
The proof is a modification of \cite[Proposition 3.9]{DL}.

{\em Step 1} -- In this step we prove that $0\in \varrho(L_\OO )$.

Since $P_\OO $ is a contraction semigroup on $L^2(\OO,\mu_\infty)$ (see Proposition \ref{prop:PKcontr}), the spectrum of $L_\OO $ is contained in the closed left-half plane.
Therefore if $0\in\sigma(L_\OO )$, it belongs to the approximate point spectrum of $L_\OO $.
This means that there is a sequence $(u_n)_{n\ge 1}$ in $\D(L_\OO )$ such that $\n u_n\n_{L^2(\OO,\mu_\infty)} = 1$
for all $n\ge 1$ and $\limn L_\OO  u_n = 0$ in $L^2(E,\mu_\infty)$.
Then,
$$\int_\OO ||D_\OO ^Hu_n||^2_H \,d\mu_{\infty}=2\int_\OO [BD_\OO ^Hu_n,D_\OO ^Hu_n]\, d\mu_{\infty} = -2[L_\OO  u_n,u_n] \rightarrow 0.$$
Hence $\limn D^H \wt u_n =0$ in  $L^2(E,\mu_\infty)$. Therefore, by Theorem \ref{thm:PoincI},
$ \limn (\wt u_n - \overline {\wt u_n}) = 0$ in $L^2(E,\mu_\infty)$.
But then $\limn \n \overline {\wt u_n}\n_{L^2(E,\mu_\infty)} = 1$, which means that
$\overline{\wt u_n} \to 1$ in $L^2(E,\mu_\infty)$.  Passing to a subsequence, we may also
assume that the convergence holds $\mu_\infty$-almost everywhere. But this contradicts
the fact that $\wt u_n$ vanishes on the set $\complement \OO$ which has positive $\mu_\infty$-measure by assumption.

{\em Step 2} -- By Step 1, $L_\OO $ is boundedly invertible,
and then $(-L_\OO )^{1/2}$ is boundedly invertible as well.
 Consequently, for $u\in \mathring{W}^{1,2}(\OO,\mui)
= \D((-L_\OO) ^{1/2}) = \D(D_\OO ^H)$ we have, by the equivalence of seminorms of Theorem 
\ref{thm:Riesz}, 
\begin{align*} \n u\n_{\LK} 
& \le \n (-L_\OO )^{-1/2}\n \n (-L_\OO) ^{1/2} u\n_{\LK}
\\ & \eqsim \n (-L_\OO )^{-1/2}\n\n D_\OO ^H u\n_{\LKH}.
\end{align*}
\epf

Note that if $\mui(\complement \OO)=0$, then we have a canonical identification 
$L^2(\OO,\mui) = L^2(E,\mui)$, and under this identification we have
$\D_\OO^H = D^H$ and $L_\OO = L$. Then it follows from Theorem \ref{thm:PoincI}
that 
$$\n u - \bar u \n_{\LK}\le C\n D_\OO ^Hu\n_{\LKH} .$$ 

\begin{remark}
 Step 1 in the above proof could be simplified (along the lines of
\cite{DL}) if we knew that $L$ has compact resolvent.
\end{remark}

{\em Acknowledgment} -- We thank the referees for their suggestions which 
improved the presentation of this paper. This work was started during a three month visit of the first named author 
at Delft University of Technology. She would like to thank the department for its kind hospitality.

\end{document}